\documentclass{amsproc}
\usepackage[pdfusetitle]{hyperref}
\usepackage{cleveref}
\usepackage{caption}
\usepackage{enumitem}
\usepackage[shortalphabetic, nobysame]{amsrefs}
\numberwithin{table}{section}
\numberwithin{equation}{section}

\hypersetup{
colorlinks=true,
linkcolor=black,
filecolor=black,
urlcolor=black,
citecolor=black,
}

\theoremstyle{plain}
\newtheorem{theorem}{Theorem}[section]
\newtheorem{prop}[theorem]{Proposition}
\theoremstyle{definition}
\newtheorem{definition}[theorem]{Definition}
\newtheorem{remark}[theorem]{Remark}
\newcommand{\negphantom}[1]{\settowidth{\dimen0}{#1}\hspace*{-\dimen0}}
\newcommand\letters{\textup{(\alph*)}}
\newcommand\CC{\mathbb{C}}
\newcommand\FF{\mathbb{F}}
\newcommand\Sym{\mathbb{S}}
\newcommand\Alt{\mathbb{A}}
\DeclareMathOperator\Aut{Aut}
\DeclareMathOperator\Out{Out}
\DeclareMathOperator\id{id}
\DeclareMathOperator\GL{GL}

\newcommand\triop{\vartriangleright}
\newcommand\OO{\mathcal{O}}
\newcommand\hC{\mathcal{C}}
\title{Twisted homogeneous racks over the alternating groups}

\author{Joseph Vulakh}
\address{Math, Science, and Technology Center, Paul Laurence Dunbar High School, Lexington, KY}
\email{joseph@vulakh.us}

\date{}

\begin{document}
\maketitle

\begin{abstract}
    An important step towards the classification of
    finite-dimensional pointed Hopf algebras is
    the classification of finite-dimensional Nichols algebras
    arising from braided vector spaces of group type.
    This question is fundamentally linked with
    the structure of algebraic objects called racks.
    Of particular interest to this classification is
    the type D condition on racks,
    a sufficient condition for a rack
    to not be the source of a finite-dimensional Nichols algebra.
    In this paper, we study the type D condition in simple racks
    arising from the alternating groups.
    Expanding upon previous work in this direction,
    we make progress towards a general classification
    of twisted homogeneous racks of type D
    by proving that several families of twisted homogeneous racks
    arising from alternating groups
    are of type D.
\end{abstract}

\section{Introduction}

The classification of finite-dimensional pointed Hopf algebras,
that is, finite-dimensional Hopf algebras with
all simple left or right comodules one-dimensional,
has been the subject of a rich line of research.
In~\cite{milinski2000pointed},
several non-trivial examples of
finite-dimensional pointed Hopf algebras were constructed
from small symmetric and dihedral groups,
and related quadratic algebras $\mathcal{E}_n$
were proved to be Nichols algebras.
The connection between pointed Hopf algebras and Nichols algebras
was made clearer by the introduction of
a classification program for finite-dimensional pointed Hopf algebras
in~\cite{andruskiewitsch2002pointed}.
This program proposes first classifying braided vector spaces $V$
of group type such that the Nichols algebra $\mathfrak{B}(V)$
is finite dimensional,
and then extending this result from Nichols algebras
to pointed Hopf algebras
constructed from $\mathfrak{B}(V)$
using the lifting method described therein.

In this paper, we consider the classification of
pointed Hopf algebras using algebraic objects called racks,
which generalize the properties of a group's conjugation action.
Racks are important for this method because
every Yetter-Drinfeld module over a finite group
can be obtained from a finite rack
through a braiding defined by $2$-cocycles associated with the rack;
this connection was uncovered in~\cite{ANDRUSKIEWITSCH2003177},
and a brief overview is given in \Cref{sec:prelim}.
%\Cref{sec:cocycles}.
As a result, the first step of the program
proposed in~\cite{andruskiewitsch2002pointed}
can be studied using racks.

Of particular importance to the classification
of finite-dimensional pointed Hopf algebras are
racks of type D, which, as explained in \Cref{sec:prelim},
%\Cref{sec:type_d},
provide a useful way of determining that certain Nichols algebras
are infinite dimensional.
It is therefore natural to attempt to classify
simple racks of type D;
this effort was initiated in~\cite{andruskiewitsch2011finite},
where most conjugacy classes of the alternating groups
and symmetric groups were determined to be of type D.
The type D condition was later investigated
for twisted homogeneous racks over the alternating groups
in~\cite{andruskiewitsch2010twisted},
for conjugacy classes of groups of Lie type
in~\cites{andruskiewitsch2011nichols, ANDRUSKIEWITSCH201536},
for twisted conjugacy calsses of groups of Lie type
in~\cite{CARNOVALE},
for conjugacy classes of the sporadic simple groups
in~\cite{ANDRUSKIEWITSCH2011305},
and for twisted conjugacy classes over the sporadic simple groups
in~\cite{fantino2013twisted}.
See \Cref{sec:prior_work} for a brief summary
of prior work as it pertains to the results of this paper.

We study the type D condition specifically for
twisted homogeneous racks over the alternating groups,
which are an important type of simple rack.
Our results build on
\cite{andruskiewitsch2010twisted}*{Theorem~1.2},
which showed that all twisted homogeneous racks not listed in
\Cref{tab:tw_hom_a_n_type_d_solved}
or \Cref{tab:tw_hom_a_n_type_d_unknown}
are of type D.
In this paper, we resolve most cases left unsolved
in~\cite{andruskiewitsch2010twisted}.
Using the notation of~\cite{andruskiewitsch2010twisted},
also explained in \Cref{sec:prior_work},
our main result is as follows.

\begin{theorem}\label{thm:main}
    For $n \geq 5$, a permutation $\ell \in \Alt_n$,
    an integer $t$ greater than $1$,
    and an automorphism $\theta$ of $\Alt_n$
    given by conjugation by an element of $\Sym_n$,
    if the twisted homogeneous rack $\hC_\ell$
    of type $(\Alt_n, t, \theta)$
    is not of type D,
    then $\hC_\ell$ is described by one of the following:
    \begin{enumerate}
        \item $\theta = \id$ and $\ell = e$,
            with $\gcd(t, n!) = 1$.
        \item $\theta = \id$ and $\ell = e$,
            with $n = 5$ and $t = 2$.
        \item $\theta = \id$ and $\ell = e$,
            with $n = 6$ and $t = 2$.
        \item $\theta = \iota_{(1\ 2)}$ and
            $\ell$ is of cycle type $(1^{s_1},2^{s_2},\dots,n^{s_n})$
            with $s_1\leq 1$ and $s_2=0$.
        \item $\theta = \iota_{(1\ 2)}$ and
        $\ell$ is of cycle type $(1^{s_1},2^{s_2},4^{s_4})$
        with $s_1\leq 2$ or $s_2\geq 1$,
        and $t = 2$.
    \end{enumerate}
\end{theorem}

\Cref{thm:main} proves that all
racks in \Cref{tab:tw_hom_a_n_type_d_solved}
are of type D,
leaving only the racks in
\Cref{tab:tw_hom_a_n_type_d_unknown}
not known to be of type D.
When $\theta = \id = \iota_e$, the only remaining cases
are twisted homogeneous racks $\hC_e$.
It was found in~\cite{andruskiewitsch2010twisted}
that the racks $\hC_e$ of type
$(\Alt_5, 2, \id)$ and $(\Alt_6, 2, \id)$
are not of type D,
and it is not known whether these racks collapse.
For the other twisted homogeneous racks with $\theta = \id$
listed in \Cref{tab:tw_hom_a_n_type_d_unknown},
it is still unknown whether they are of type D.
The cases remaining when $\theta = \iota_{(1\ 2)}$
are more broad, and include racks $\hC_\ell$
with $\ell (1\ 2)$ of arbitrarily large order.
In addition, the case when
$\theta$ is an exceptional automorphism of $\Alt_6$
has not yet been explored.

We discuss the technical background of our work
in \Cref{sec:prelim},
where we provide definitions of the algebraic structures
used in this paper.
We review prior work on
simple racks, and
twisted homogeneous racks over the alternating groups
in particular,
in \Cref{sec:prior_work}.
Finally, in \Cref{sec:tw_hom_racks},
we prove our main result, \Cref{thm:main}.

\section{Preliminaries}\label{sec:prelim}

%\subsection{Preliminaries}\label{sec:prelim}

We denote the group identity by $e$ throughout this paper.
In this section,
we summarize the technical preliminaries of our work.

%For a group $G$ acting on a set $S$,
%we denote the orbit of an element $x$ of $S$ under the action of $G$
%by $\OO_x^G$.
%In particular, the conjugacy class of an element of $x$ in $G$ is
%exactly the orbit of $x$ when $G$ acts on itself by conjugation,
%so we denote the conjugacy class of element $x$ of $G$
%by $\OO_x^G$.

%\subsection{Racks}

We first describe basic notions related to racks.
See~\cite{ANDRUSKIEWITSCH2003177}
and~\cite{andruskiewitsch2011finite}
for more details.
A \emph{rack} is a pair $(X, \triop)$,
with $X$ a non-empty set and
$\triop \colon X \times X \rightarrow X$ a function,
such that:
\begin{enumerate}
    \item For all $x$ in $X$,
    the map $\varphi_x \colon X \rightarrow X$
    given by
    \begin{equation*}
        \varphi_x (y) = x \triop y
    \end{equation*}
    is a bijection.
    \item For all $x$, $y$, $z$ in $X$,
    \begin{equation*}
        x \triop (y \triop z) = (x \triop y) \triop (x \triop z).
    \end{equation*}
\end{enumerate}

Additional definitions concerning racks are listed below.
\begin{itemize}
    \item A \emph{morphism} of racks is a function
        $f \colon X \rightarrow Y$
        such that
        \begin{equation*}
            f(x \triop y) = f(x) \triop f(y)
        \end{equation*}
        for all $x$, $y$ in $X$.
    \item A \emph{subrack} of a rack $X$ is a non-empty subset $Y$
        such that $Y \triop Y = Y$.
    \item A \emph{decomposition} of a rack $X$ is a family
        ${(X_i)}_{i \in I}$ of pairwise disjoint subracks of $X$
        such that $X = \coprod_{i \in I} X_i$ and
        $X \triop X_i = X_i$ for all $i$ in $I$.
    \item A rack $X$ is \emph{decomposable} if
        there exists a decomposition of $X$.
    \item A rack is \emph{trivial} if it contains exactly one element.
\end{itemize}

An important construction of racks comes from group conjugation.
Let $G$ be a group,
and let $x \triop y = xyx^{-1}$ for $x$, $y$ in $G$.
Then $(G, \triop)$ is a rack,
and conjugacy classes of $G$ are subracks.

More generally, we consider twisted conjugacy classes
of the group $G$. Let $f \in \Aut(G)$,
and consider the action of $G$ on itself given by
$y \rightharpoonup_f x = yx f(y^{-1})$.
For $x$ in $G$,
the orbit $\OO_x^{G, f}$ of $x$ under this action
is called the \emph{twisted conjugacy class} of $x$
and is a rack with operation
$y \triop_f z = y f(zy^{-1})$.
We will refer to these racks as
\emph{twisted conjugacy classes of type $(G, f)$}.

Another important rack construction is that of an \emph{affine rack}.
Let $A$ be a finite abelian group,
and let $g \colon A \rightarrow A$ be an automorphism.
Define the affine rack $(A, g)$ as
the rack with underlying set $A$
and rack operation
$x \triop y = x + g(y - x)$.
Letting $f = \id - g$, this can be written as
$x \triop y = f(x) + g(y)$ or $x \triop y = f(x - y) + y$.

%\subsection{Cocycles and Nichols Algebras}\label{sec:cocycles}

We now consider cocycles and Nichols algebras.
Let ${(X_i)}_{i \in I}$ be a decomposition of a rack $X$,
and let $\mathbf{n} = {(n_i)}_{i \in I}$ be
a family of positive integers.
A \emph{$2$-cocycle of degree $\mathbf{n}$} is a family
$\mathbf{q} = {(q_i)}_{i \in I}$
of functions $q_i \colon X \times X_i \rightarrow \GL(n_i, \CC)$
satisfying
\begin{equation*}
    q_i (x, y \triop z) q_i (y, z) =
    q_i (x \triop y, x \triop z) q_i (x, z)
\end{equation*}
for all $i \in I$, $x, y \in X$, $z \in X_i$.

Let $V = \bigoplus_{i \in I} \CC X_i \otimes \CC^{n_i}$,
and denote the element of $\CC X_i$
corresponding to an element $x$ of $X_i$
by $e_x$.
For a family $\mathbf{q} = {(q_i)}_{i \in I}$
of functions $q_i \colon X \times X_i \rightarrow \GL(n_i, \CC)$,
consider the linear map
$c^\mathbf{q} \colon V \otimes V \rightarrow V \otimes V$
acting on the basis elements of $V \otimes V$ by
\begin{equation*}
    c^\mathbf{q} (e_x v \otimes e_y w)
    = e_{x \triop y} q_i (x, y) (w) \otimes e_x v
\end{equation*}
for $x \in X_j$, $y \in X_i$, $v \in \CC^{n_j}$, $w \in \CC^{n_i}$.
By~\cite{ANDRUSKIEWITSCH2003177}*{Theorem~4.14},
the map $c^\mathbf{q}$ satisfies the braid equation
\begin{equation*}
    (\id \otimes c^\mathbf{q})
    (c^\mathbf{q} \otimes \id)
    (\id \otimes c^\mathbf{q}) =
    (c^\mathbf{q} \otimes \id)
    (\id \otimes c^\mathbf{q})
    (c^\mathbf{q} \otimes \id)
\end{equation*}
if and only if $\mathbf{q}$ is a $2$-cocycle.
Moreover, if $\mathbf{q}$ is a $2$-cocycle,
then there exists a group $G$ such that $V$ is a
Yetter–Drinfeld module over $G$, and any Yetter-Drinfeld module
over a finite group can be realized in this way.
The Nichols algebra of the Yetter-Drinfeld module obtained in this way
is denoted $\mathfrak{B} (X, \mathbf{q})$.

Let $X$ be a rack, $\mathbf{q}$ a $2$-cocycle,
and $V$ a vector space as above.
The $2$-cocycle $\mathbf{q}$ is \emph{finite} if the image of $q_i$
generates a finite subgroup of $\GL(n_i, \CC)$ for all $i \in I$.
Let $g \colon X \rightarrow \GL(V)$ be the map given by
\begin{equation*}
    g_x (e_y w) = e_{x \triop y} q_i (x, y) (w)
\end{equation*}
for $i \in I$, $x \in X$, $y \in X_i$, $w \in \CC^{n_i}$.
Using the rack structure of $\GL(V)$ given by conjugation,
the map $g$ is a rack morphism.
The cocycle $\mathbf{q}$ is \emph{faithful} if $g$ is injective.
These notions are important for the following definition.

\begin{definition}[\cite{andruskiewitsch2011finite}*{Definition~2.2}]\label{def:collapse}
    A finite rack $X$ \emph{collapses} if for every
    finite faithful cocycle $\mathbf{q}$ associated with
    any decomposition of $X$ and any degree $\mathbf{n}$,
    the Nichols algebra $\mathfrak{B} (X, \mathbf{q})$ is
    infinite-dimensional.
\end{definition}

%\subsection{Racks of Type D}\label{sec:type_d}

%We now introduce
We close this section by introducing
the type D condition,
which is the main focus of this paper.

\begin{definition}{\cite{andruskiewitsch2011finite}*{Definition~3.5}}\label{def:type_d}
    A rack $X$ is of \emph{type D} if
    there exists a decomposable subrack $Y = R \coprod S$ of $X$
    such that
    \begin{equation}\label{eq:type_d}
        r \triop (s \triop (r \triop s) ) \neq s
    \end{equation}
    for some $r \in R$ and $s \in S$.
\end{definition}

\begin{remark}\label{rem:type_d_conj}
    When $X$ is the rack arising from a conjugacy class
    $\OO_x^G$ of a group $G$,
    \Cref{eq:type_d} is equivalent to
    the condition that ${(rs)}^2 \neq {(sr)}^2$.
    The condition that $R$ and $S$ are disjoint
    translates into the condition that
    $r$ and $s$ are not conjugates
    in the group generated by $r$ and $s$.
\end{remark}

The significance of this definition stems from the following result,
which, loosely speaking, states that any Nichols algebra
arising from a rack of type D is infinite-dimensional,
allowing for a rack to be discarded from consideration
as being the source of a finite-dimensional Nichols algebra.

\begin{theorem}[\cite{andruskiewitsch2011finite}*{Theorem~3.6}]
    If $X$ is a finite rack of type D, then
    $X$ collapses.
\end{theorem}

%\subsection{Simple Racks over the Alternating Groups}
\section{Prior Work}\label{sec:prior_work}

Of particular research interest are those racks which have no
nontrivial quotients. We consider such racks in this section.

A rack $X$ is \emph{simple} if it is not trivial
and for any rack morphism $f \colon X \rightarrow Y$,
either $|Y| = 1$ or $|Y| = |X|$.
Simple racks have been classified
in~\cite{ANDRUSKIEWITSCH2003177}*{Theorem~3.9--Theorem~3.12}.

\begin{theorem}
[\cite{ANDRUSKIEWITSCH2003177}*{Theorem~3.9--Theorem~3.12}]
\label{thm:simple_rack_classif}
    For a simple rack $X$, exactly one of the following holds:
    \begin{enumerate}
        \item $|X| = p$, where $p$ is prime,
            and $X \cong \FF_p$, with $x \triop y = y + 1$.
            Such racks are called \emph{permutation racks}.
        \item $|X| = p^t$, with $p$ a prime and $t$ a positive integer,
            and $X \cong (\FF_p^t, T)$ an affine rack,
            where $T$ is the companion matrix of
            an irreducible monic polynomial in $\FF_p[X]$ of degree $t$
            different from $X$ and $X - 1$.
        \item $|X|$ is divisible by at least two distinct primes,
            and there exists a non-abelian simple group $L$,
            a positive integer $t$,
            and an automorphism $\theta$ in $\Aut(L)$
            such that $X$ is a twisted conjugacy class of $L^t$
            with respect to the automorphism
            \begin{equation*}
                x \colon L^t \rightarrow L^t, \quad
                x(\ell_1, \dots, \ell_t)
                = (\theta(\ell_t), \ell_1, \dots, \ell_{t - 1}).
            \end{equation*}
            Racks produced by such a construction are called
            \emph{twisted homogeneous racks}.
            In the case when $t = 1$,
            the rack $X$ is a twisted conjugacy class of a simple group.
    \end{enumerate}
\end{theorem}

We refer to simple racks as in (3) as
\emph{twisted homogeneous racks of type $(L, t, \theta)$}.
Following~\cite{andruskiewitsch2010twisted}, we denote
the twisted homogeneous rack of $(e, \dots, e, \ell)$ in $L^t$
by $\hC_\ell$.
The power of this notation stems from the following theorem.

\begin{theorem}
[\cite{andruskiewitsch2010twisted}*{Proposition 3.3}]
\label{thm:tw_hom_rack_prod}
    An element $(x_1, \dots, x_t)$ of $L^t$ is in $\hC_\ell$
    if and only if
    $x_t x_{t - 1} \cdots x_2 x_1 \in \OO_\ell^{L, \theta}$.
\end{theorem}

\Cref{thm:tw_hom_rack_prod} establishes a bijection
between twisted homogeneous racks of type $(L, t, \theta)$
and twisted conjugacy classes of $L$
and allows twisted homogeneous racks to be identified
by a single element of $L$.

In this paper, we are interested in simple racks
over the alternating groups.
By \Cref{thm:simple_rack_classif},
the twisted homogeneous racks of type
$(\Alt_n, t, \theta)$ with $n \geq 5$ are simple.
It is well-known that, for $n \neq 6$,
the automorphisms of $\Alt_n$ are given by conjugation
by an element of $\Sym_n$
and $\Out(\Alt_n) \cong C_2$
(see, for example, \cite{suzuki1982group}*{Chapter~3.2}).
We therefore consider only the automorphisms
$\theta = \id = \iota_e$ and $\theta = \iota_{(1\ 2)}$,
where $\iota_{u}$ acts on $\Alt_n$ by conjugation by $u$.

The case of $t = 1$ has been considered
in~\cite{andruskiewitsch2011finite}
and~\cite{andruskiewitsch2011nichols}.
Set $u$ to equal $e$ if $\theta = \id$
and $(1\ 2)$ if $\theta = \iota_{(1\ 2)}$.
It is known that
if the cycle type of $xu$ is not in the following list,
then $\OO_x^{G, \theta}$ is of type D
\cite{andruskiewitsch2011finite}*{Theorem~4.1}
(see also \cite{andruskiewitsch2011nichols}*{Theorem~6.1}):
\begin{enumerate}[label=\letters]
    \item $(2, 3)$; $(2^3)$; $(1^k, 2)$.
    \item $(3^2)$; $(2^2, 3)$; $(1^k, 3)$; $(2^4)$;
        $(1^2, 2^2)$; $(1, 2^2)$;
        $(1, p)$, $(p)$, where $p$ is prime.
\end{enumerate}
Even if $xu$ has cycle type in list (b), $\OO_x^{G, \theta}$
is known to collapse~\cite{andruskiewitsch2011finite}*{Theorem~1.1}.

For the rest of this paper, we consider only the case
of $t > 1$,
which was first investigated in~\cite{andruskiewitsch2010twisted}.
Many twisted homogeneous racks over the alternating groups
were shown to be of type D
in~\cite{andruskiewitsch2010twisted}*{Theorem~1.2},
but many racks, listed in
\Cref{tab:tw_hom_a_n_type_d_solved}
and \Cref{tab:tw_hom_a_n_type_d_unknown},
remained not known to be of type D.
Interestingly, two twisted homogeneous racks,
listed as part of \Cref{tab:tw_hom_a_n_type_d_unknown},
were found \emph{not} to be of type D
in~\cite{andruskiewitsch2010twisted},
and it is not known whether these racks collapse.

\Cref{thm:main} resolves all cases listed in
\Cref{tab:tw_hom_a_n_type_d_solved}.

\begin{table}[ht]
\begin{center}
\begin{tabular}{|c|c|c|c|c|}
    \hline
    $u$ & $n$ & Type of $\ell u$ & $t$ & Resolved by\\
    \hline
    $e$ & 5 & $(1^5)$ & $4$ & \Cref{prop:id_1} \\
    \cline{2-5}
    & $5$ & $(1,2^{2})$ & $4$, odd & \Cref{prop:id_12r} \\
    \cline{2-4}
    & 6 & $(1^{2},2^{2})$ & odd & \\
    \cline{2-4}
    & 8 & $(2^{4})$ & odd & \\
    \cline{2-5}
    & any & $(1^{r_1},2^{r_2},4^{r_4})$ & $2$
        & \Cref{prop:id_124},\negphantom{,} \\
    && with $r_4>0$, $r_2+r_4$ even && \Cref{prop:id_14} \\
    \hline
    $(1\ 2)$ & $5$ & $(1^3,2)$ & $2$, $4$
        & \Cref{prop:iota_12} \\
    \cline{2-5}
    & 6 & $(1^4,2)$ & $2$ & \Cref{prop:iota_12} \\
    \cline{2-5}
    & 6 & $(2^3)$ & 2 & \Cref{prop:iota_222} \\
    \cline{2-5}
    & 7 & $(1,2^3)$ & $2$, odd & \Cref{prop:iota_12r} \\
    \cline{2-4}
    & 8 & $(1^2,2^3)$ & odd & \\
    \cline{2-4}
    & 10 & $(2^5)$ &  odd &  \\
    \hline
\end{tabular}
\caption{Twisted homogeneous racks $\hC_\ell$
    of type $(\Alt_n, t, \iota_u)$,
    where $t > 1$ and $n \geq 5$,
    shown to be of type D
    by \Cref{thm:main}.
}
\label{tab:tw_hom_a_n_type_d_solved}
\end{center}
\end{table}

\begin{table}[ht]
\begin{center}
\begin{tabular}{|c|c|c|c|c|}
    \hline
    $u$ & $n$ & Type of $\ell u$ & $t$ & Type D \\
    \hline
    $e$ & any & $(1^n)$ & $\gcd(t,n!) = 1$ & unknown \\
        & 5 & $(1^5)$ & $2$ & not type D \\
        & 6 & $(1^6)$ & $2$ & not type D \\
    \hline
    $(1\ 2)$ & any & $(1^{s_1},2^{s_2},\dots,n^{s_n})$
        with $s_1\leq 1$, $s_2=0$, & any & unknown \\
    && $s_h\geq 1$ for some $h$ with $3\leq h\leq n$ && \\
    &&&&\\
    && $(1^{s_1},2^{s_2},4^{s_4})$
        with $s_1\leq 2$ or $s_2\geq 1$, & $2$ & unknown \\
    && $s_2+s_4$ odd, $s_4\geq 1$ && \\
    \hline
\end{tabular}
\caption{Twisted homogeneous racks $\hC_\ell$
    of type $(\Alt_n, t, \iota_u)$, where $t > 1$ and $n \geq 5$,
    not known to be of type D.
    Two of these racks were found not to be of type D
    in~\cite{andruskiewitsch2010twisted}.
}
\label{tab:tw_hom_a_n_type_d_unknown}
\end{center}
\end{table}

\section{Twisted Homogeneous Racks}
\label{sec:tw_hom_racks}

In this section,
we prove \Cref{thm:main},
considering first the twisted homogeneous racks
arising from the identity automorphism
in \Cref{sec:id},
and then those arising from $\iota_{(1\ 2)}$,
the automorphism given by conjugation by $(1\ 2)$,
in \Cref{sec:iota}.

\subsection{Identity Automorphism}
\label{sec:id}

In this section, we prove that the racks listed in
the first section of \Cref{tab:tw_hom_a_n_type_d_solved}
are of type D.
We begin with the case when $\ell = e$.

\begin{prop}\label{prop:id_1}
    For even $t$ greater than or equal to $4$
    and $n \geq 5$,
    the twisted homogeneous rack
    $\hC_e$ of type $(\Alt_n, t, \id)$
    is of type D.
\end{prop}

\begin{proof}
    Let $G = \{ e, (1\ 2\ 3), (1\ 3\ 2),
    (1\ 2)(4\ 5), (2\ 3)(4\ 5), (1\ 3)(4\ 5) \}$.
    The group $G$ is isomorphic to $\Sym_3$,
    with the isomorphism $\varphi$ given by
    considering the action of $G$ only on $\{1, 2, 3\}$.

    Let $R$ be the set of $t$-tuples $(g_1, \dots, g_t)$
    of elements of $G$
    with $g_t g_{t - 1} \cdots g_2 g_1 = e$
    and each $\varphi(g_i)$ an even permutation,
    and let $S$ be the set of $t$-tuples $(g_1, \dots, g_t)$
    of elements of $G$
    with $g_t g_{t - 1} \cdots g_2 g_1 = e$
    and each $\varphi(g_i)$ an odd permutation.
    Let $Y = R \cup S$.
    The set $Y$ is a subset of $\hC_e$
    by \Cref{thm:tw_hom_rack_prod},
    and because $t$ is even, $R$ and $S$ are nonempty.
    Moreover, for two elements $a = (a_1, \dots, a_t)$
    and $b = (b_1, \dots, b_t)$
    of $Y$,
    \begin{align*}
        a \triop b &= (a_1, \dots, a_t) \triop (b_1, \dots, b_t) \\
        &= \left( a_1 b_t a_t^{-1},
        a_2 b_1 a_1^{-1}, \dots, a_t b_{t - 1} a_{t - 1}^{-1} \right),
    \end{align*}
    each element of which has the same parity
    upon mapping through $\varphi$
    as the corresponding $b_i$.
    Also, the product of the tuple elements of $a \triop b$ is
    \begin{align*}
        (a_t b_{t - 1} a_{t - 1}^{-1})
        (a_{t - 1} b_{t - 2} a_{t - 2}^{-1}) \cdots
        (a_2 b_1 a_1^{-1}) (a_1 b_t a_t^{-1})
        &= a_t (b_{t - 1} b_{t - 2} \cdots b_1) b_t a_t^{-1} \\
        &= a_t b_t^{-1} b_t a_t^{-1} \\
        &= e,
    \end{align*}
    so the sets $Y$, $R$, and $S$ are closed under the rack operation.
    Therefore, $Y$, $R$, and $S$ are racks,
    and $Y = R \coprod S$.

    Let $x = (1\ 2)(4\ 5)$ and $y = (2\ 3)(4\ 5)$.
    Now let $r = e \in R$ and
    $s = (x, x, y, y, \dots) \in S$,
    with each element of $s$ except for the first two equal to
    $y$.
    Then \begin{align*}
        r \triop (s \triop (r \triop s))
        &= r \triop (s \triop (y, x, x, y, \dots)) \\
        &= r \triop (x, xyx^{-1}, y, yxy^{-1}, \dots) \\
        &= \begin{cases}
            (yxy^{-1}, x, xyx^{-1}, y), &\text{if } t = 4 \\
            (y, x, xyx^{-1}, y, yxy^{-1}, \dots) &\text{if } t > 4
        \end{cases} \\
        &\neq s,
    \end{align*}
    so $\hC_\ell$ is of type D.
\end{proof}

We now consider the case when $\ell$ is an involution.

\begin{prop}\label{prop:id_12r}
    For $\ell$ in $\Alt_n$ and of cycle type $(1^{r_1}, 2^{r_2})$,
    with $r_2$ even and positive,
    the twisted homogeneous rack $\hC_\ell$ of type
    $(\Alt_n, t, \id)$, with $t \geq 3$, is of type D.
%    For $\ell$ of cycle type $(1^{n - 4}, 2^2)$ and $t \geq 3$,
%    the twisted homogeneous rack $\hC_\ell$
%    of type $(\Alt_n, t, id)$ is of type D.
\end{prop}

\begin{proof}
    Let
    \begin{align*}
        x &= \prod\limits_{i = 1}^{\frac{r_2}{2}}
        (4i - 3\ 4i - 2)(4i - 1\ 4i) \\
        y &= \prod\limits_{i = 1}^{\frac{r_2}{2}}
        (4i - 3\ 4i - 1)(4i - 2\ 4i) \\
        z &= \prod\limits_{i = 1}^{\frac{r_2}{2}}
        (4i - 3\ 4i)(4i - 2\ 4i - 1).
    \end{align*}
    Then $\{e, x, y, z\}$ is isomorphic to the Klein four-group,
    which is abelian.
    Let $Y$ be the set
    consisting of $t$-tuples $(c_1, \dots, c_t)$
    with each $c_i$ in $\{e, x, y, z\}$
    and with $\prod_{i = 1}^t c_i$ equal to $x$ or $y$.
    By \Cref{thm:tw_hom_rack_prod}, $Y$ is a subset of $\hC_\ell$.
    Because the Klein four-group is abelian,
    the rack action in $Y$ does not change
    the product of the elements of a tuple in $Y$.
    Therefore, letting $R$ be the subset of $Y$
    consisting of tuples $(c_1, \dots, c_t)$
    with $\prod_{i = 1}^t c_i = x$
    and $S$ be the subset of $Y$
    consisting of tuples $(c_1, \dots, c_t)$
    with $\prod_{i = 1}^t c_i = y$,
    we find that $R$, $S$, and $Y = R \coprod S$ are racks.
    
    Now let $r = (x, e, \dots) \in R$
    and $s = (y, e, \dots) \in S$.
    These elements satisfy \Cref{eq:type_d} when $t \geq 3$.
    Indeed,
    \begin{align*}
        s \triop (r \triop s) &=
        s \triop (x, z, e, \dots) \\
        &= (y, z, z, \dots),
    \end{align*}
    so $r \triop (s \triop (r \triop s))$ is equal to
    $(y, z, z)$ if $t = 3$,
    or to $(x, z, z, z, \dots)$ if $t \geq 4$.
    In both cases, \Cref{eq:type_d} is satisfied.
\end{proof}

\Cref{prop:id_12r} resolves the first three rows
of \Cref{tab:tw_hom_a_n_type_d_solved}.
In particular, it proves that racks $\hC_\ell$
of type $(\Alt_n, t, \id)$
with $n$ equal to $5$ or $6$,
the integer $t$ equal to $4$ or odd and greater than or equal to $3$,
and $\ell$ an involution of type $(1, 2^2)$ or $(1^2, 2^2)$
and racks $\hC_\ell$
of type $(\Alt_8, t, \id)$
with $t$ odd and greater than or equal to $3$
and $\ell$ an involution of type $(2^4)$
are all of type D.
We proceed to resolve the last row of
the first section of \Cref{tab:tw_hom_a_n_type_d_solved}.

\begin{prop}\label{prop:id_124}
    For $\ell$ in $\Alt_n$ and of cycle type
    $(1^{r_1}, 2^{r_2}, 4^{r_4})$, with $r_2, r_4 > 0$,
    the twisted homogeneous rack $\hC_\ell$ of type
    $(\Alt_n, 2, \id)$ is of type D.
\end{prop}

\begin{proof}
    Let $x = (1\ 2)(3\ 4\ 5\ 6)$ and $y = (1\ 2\ 3\ 6)(4\ 5)$,
    and let $G$ be the subgroup of $\Alt_n$ generated by $x$ and $y$.
    The action of $G$ on the set
    $S = \{ (1\ 3\ 5),\allowbreak
    (1\ 5\ 3),\allowbreak
    (2\ 4\ 6),\allowbreak
    (2\ 6\ 4) \}$
    by conjugation
    induces a homomorphism $\varphi$
    into the cyclic group $C_4$,
    with $x$ and $y$ mapping to distinct generators.
    Indeed, the action of $x$ is given by the cycle
    $((1\ 3\ 5)\allowbreak\ 
    (2\ 4\ 6)\allowbreak\ 
    (1\ 5\ 3)\allowbreak\ 
    (2\ 6\ 4))$,
    and the action of $y$ is given by the cycle
    $((1\ 3\ 5)\allowbreak\ 
    (2\ 6\ 4)\allowbreak\ 
    (1\ 5\ 3)\allowbreak\ 
    (2\ 4\ 6))$.

    Let $\pi$ be an element of $\Alt_n$ acting only on numbers
    greater than $6$ and with cycle type
    $(1^{r_1}, 2^{r_2 - 1}, 4^{r_4 - 1})$.
    Now let $R$ be the set consisting of pairs $(r_1, r_2)$
    such that $r_1 \pi^{-1}$ and $r_2$ are in $G$
    and $r_1 \pi^{-1}$ and $r_2$ have product mapped
    to $\varphi(x)$ by $\varphi$.
    All elements of $\Alt_6$ which act on $S$
    in the same way $x$ does have cycle type $(2, 4)$,
    so $R$ is a subset of $\hC_\ell$
    by \Cref{thm:tw_hom_rack_prod}.
    Also, the rack operation in $R$ is of the form
    $(g_1 \pi, h_1) \triop (g_2 \pi, h_2)
    = (g_1 h_2 h_1^{-1} \pi, h_1 g_2 g_1^{-1})$
    and $\varphi(g_1 h_2 g_2 g_1^{-1}) = \varphi(g_2 h_2)$,
    so $R$ is closed under the rack operation and hence is a rack.
    Similarly, let $S$ be the subrack of $\hC_\ell$
    consisting of pairs $(s_1, s_2)$
    such that $s_1 \pi^{-1}$ and $s_2$ are in $G$
    and $s_1 \pi^{-1}$ and $s_2$ have product mapping to $\varphi(y)$.

    The set $Y = R \coprod S$ is a subrack of $\hC_\ell$.
    Finally,
    \begin{align*}
        (x\pi, e) \triop (y\pi, e)
        &= (x\pi, yx^{-1}) \\
        (y\pi, e) \triop (x\pi, yx^{-1})
        &= (y^2x^{-1}\pi, xy^{-1}) \\
        (x\pi, e) \triop (y^2x^{-1}\pi, xy^{-1})
        &= (x^2y^{-1}\pi, y^2x^{-2}) \\
        &\neq (y\pi, e),
    \end{align*}
    with $(x\pi, e) \in R$ and $(y\pi, e) \in S$,
    so $\hC_\ell$ is of type D.
\end{proof}

\Cref{prop:id_124} covers all $\ell$
included within the last row of
the first section of \Cref{tab:tw_hom_a_n_type_d_solved}
satisfying $r_2 > 0$.
Because $r_4$ must be even in order for $\ell$ to be in $\Alt_n$,
the only case remaining for this row is when
$r_2 = 0$ and $r_4 > 1$,
addressed in the following proposition.

\begin{prop}\label{prop:id_14}
    For $\ell$ in $\Alt_n$ and of cycle type $(1^{r_1}, 4^{r_4})$,
    with $r_4 > 1$, the twisted homogeneous rack $\hC_\ell$
    of type $(\Alt_n, 2, \id)$ is of type D.
\end{prop}

\begin{proof}
    This follows from the same reasoning as \Cref{prop:id_124},
    but with $x = (1\ 2\ 3\ 4)(5\ 6\ 7\ 8)$
    and $y = (1\ 6\ 7\ 8)(2\ 3\ 4\ 5)$
    acting on the set
    $\{ (1\ 3\ 5\ 7),\allowbreak
    (1\ 7\ 5\ 3),\allowbreak
    (2\ 4\ 6\ 8),\allowbreak
    (2\ 8\ 6\ 4) \}$.
    The element $x$ acts by the cycle
    $((1\ 3\ 5\ 7)\allowbreak\ 
    (2\ 4\ 6\ 8)\allowbreak\ 
    (1\ 7\ 5\ 3)\allowbreak\ 
    (2\ 8\ 6\ 4))$,
    and the element $y$ acts by the cycle
    $((1\ 3\ 5\ 7)\allowbreak\ 
    (2\ 8\ 6\ 4)\allowbreak\ 
    (1\ 7\ 5\ 3)\allowbreak\ 
    (2\ 4\ 6\ 8))$,
    giving the needed homomorphism $\varphi \colon G \rightarrow C_4$.
\end{proof}

\subsection{Automorphism by Conjugation by a Transposition}
\label{sec:iota}

Recall that the automorphism of $\Alt_n$
acting by conjugation by $(1\ 2)$
is denoted $\iota_{(1\ 2)}$.
In this section, we prove that
the twisted homogeneous racks in
the second section of \Cref{tab:tw_hom_a_n_type_d_solved}
are of type D.
We begin with the case when $\ell(1\ 2)$ is a transposition.

\begin{prop}\label{prop:iota_12}
    For even $t$ and $n \geq 5$,
    the twisted homogeneous rack $\hC_e$
    of type $(\Alt_n, t, \iota_{(1\ 2)})$
    is of type D.
\end{prop}

\begin{proof}
    Let $G = \{ e, (1\ 2\ 3), (1\ 3\ 2),
    (1\ 2)(4\ 5), (2\ 3)(4\ 5), (1\ 3)(4\ 5) \}$.
    As observed in the proof of \Cref{prop:id_1},
    the group $G$ is isomorphic to $\Sym_3$,
    with the isomorphism $\varphi$ given by
    considering the action of $G$ only on $\{1, 2, 3\}$.

    Let $R$ be the set of $t$-tuples of elements $g$ of $G$
    with $\varphi(g)$ an even permutation,
    let $S$ be the set of $t$-tuples of elements $g$ of $G$
    with $\varphi(g)$ an odd permutation,
    and let $Y = R \cup S$.
    Because each element of $G$ with even image under $\varphi$
    is in $\OO_e^{\Alt_n, \iota_{(1\ 2)}}$ and $t$ is even,
    $Y$ is a subset of $\hC_e$
    by \Cref{thm:tw_hom_rack_prod}.
    Moreover, for any two elements $a = (a_1, \dots, a_t)$
    and $b = (b_1, \dots, b_t)$
    of $Y$,
    \begin{align*}
        a \triop b &= (a_1, \dots, a_t) \triop (b_1, \dots, b_t) \\
        &= \left( a_1 (1\ 2) b_t a_t^{-1} (1\ 2),
        a_2 b_1 a_1^{-1}, \dots, a_t b_{t - 1} a_{t - 1}^{-1} \right),
    \end{align*}
    each element of which has the same parity
    upon mapping through $\varphi$
    as the corresponding $b_i$.
    Therefore, $Y$, $R$, and $S$ are racks,
    and $Y = R \coprod S$.

    Finally, let $r = e$ and $s = ((1\ 3)(4\ 5), \dots, (1\ 3)(4\ 5))$
    in $\hC_e$.
    Then
    \begin{align*}
        r \triop (s \triop (r \triop s)) &=
        r \triop (s \triop
        ((2\ 3)(4\ 5), (1\ 3)(4\ 5), \dots) \\
        &= r \triop ((1\ 3)(4\ 5), (1\ 2)(4\ 5), \dots) \\
        &= \begin{cases}
            ((1\ 2)(4\ 5), (1\ 3)(4\ 5)), &\text{if } t = 2 \\
            ((2\ 3)(4\ 5), (1\ 3)(4\ 5), (1\ 2)(4\ 5),
            (1\ 3)(4\ 5), \dots) &\text{if } t > 2
        \end{cases} \\
        &\neq s,
    \end{align*}
    so $\hC_\ell$ is of type D, as claimed.
\end{proof}

We now address the case when $\ell(1\ 2)$ is a product of
a larger number of transpositions,
beginning with the case $t = 2$.

\begin{prop}\label{prop:iota_222}
    For $\ell(1\ 2)$ in $\Sym_n$
    and of type $(2^3)$,
    the twisted homogeneous rack $\hC_\ell$
    of type $(\Alt_n, 2, \iota_{(1\ 2)})$ is of type D.
\end{prop}

\begin{proof}
    Let $R$ and $S$ be the subracks of $\hC_\ell$
    defined by \Cref{tab:iota_222},
    with each row corresponding to one element of each rack.
    Each row determines an element of each rack;
    the first column contains the first permutation of the pair,
    and the column corresponding to the rack (either $R$ or $S$)
    contains the second permutation of the pair.

    The racks $R$ and $S$ can also be described as follows.
    Let $x = (3\ 4)(5\ 6)$ and $y = (3\ 6)(4\ 5)$.
    For each element $\pi$ of the alternating group
    acting on $\{3, 4, 5, 6\}$,
    let the element $(\pi, x \pi^{-1} x)$
    be in $R$ if $x$ is not an involution,
    and let the element $(\pi, x \pi)$
    be in $R$ if $x$ is an involution.
    The elements of $S$ are characterized similarly using $y$.
    
    The sets $R$ and $S$ so defined
    can be checked to be racks
    with $Y = R \coprod S$ also a rack.
    Moreover, setting $r = ((3\ 4)(5\ 6), e)$
    and $s = ((3\ 4\ 5), (4\ 5\ 6))$,
    we obtain
    \begin{align*}
        r \triop (s \triop (r \triop s)) &=
        r \triop (s \triop ((3\ 4\ 6), (3\ 5\ 6))) \\
        &= r \triop (e, (3\ 6)(4\ 5)) \\
        &= ((3\ 5)(4\ 6), (3\ 4)(5\ 6)) \\
        &\neq s,
    \end{align*}
    so $\hC_\ell$ is of type D.
\end{proof}

\begin{table}[t]
    \begin{center}
        \begin{tabular}{|c|c|c|}
            \hline
            First permutation & \multicolumn{2}{c|}{Second permutation}
            \\\hline
            & $R$ & $S$ \\\hline
            $e$ & $(3\ 4)(5\ 6)$ & $(3\ 6)(4\ 5)$ \\\hline
            $(4\ 5\ 6)$ & $(3\ 5\ 6)$ & $(3\ 4\ 5)$ \\\hline
            $(4\ 6\ 5)$ & $(3\ 6\ 5)$ & $(3\ 5\ 4)$ \\\hline
            $(3\ 4)(5\ 6)$ & $e$ & $(3\ 5)(4\ 6)$ \\\hline
            $(3\ 4\ 5)$ & $(3\ 4\ 6)$ & $(4\ 5\ 6)$ \\\hline
            $(3\ 4\ 6)$ & $(3\ 4\ 5)$ & $(3\ 5\ 6)$ \\\hline
            $(3\ 5\ 4)$ & $(3\ 6\ 4)$ & $(4\ 6\ 5)$ \\\hline
            $(3\ 5\ 6)$ & $(4\ 5\ 6)$ & $(3\ 4\ 6)$ \\\hline
            $(3\ 5)(4\ 6)$ & $(3\ 6)(4\ 5)$ & $(3\ 4)(5\ 6)$ \\\hline
            $(3\ 6\ 4)$ & $(3\ 5\ 4)$ & $(3\ 6\ 5)$ \\\hline
            $(3\ 6\ 5)$ & $(4\ 6\ 5)$ & $(3\ 6\ 4)$ \\\hline
            $(3\ 6)(4\ 5)$ & $(3\ 5)(4\ 6)$ & $e$ \\\hline
        \end{tabular}
        \caption{Racks $R$ and $S$ used in the proof
            of \Cref{prop:iota_222}.
            Each row describes an element of $R$
            and an element of $S$:
            if the permutations, reading across,
            are $\pi$, $\sigma$, and $\tau$,
            then $(\pi, \sigma) \in R$ and $(\pi, \tau) \in S$.}
        \label{tab:iota_222}
    \end{center}
\end{table}

We now address the case $t \geq 3$.

\begin{prop}\label{prop:iota_12r}
    For $\ell (1\ 2)$ in $\Sym_n$
    and of type $(1^{r_1}, 2^{r_2})$,
    with $r_2$ odd and greater than $1$,
    the twisted homogeneous rack $\hC_\ell$
    of type $(\Alt_n, t, \iota_{(1\ 2)})$,
    with $t \geq 3$,
    is of type D.
\end{prop}

\begin{proof}
    Because $\iota_{(1\ 2)}$ leaves permutations
    which act trivially on $\{1, 2\}$ fixed,
    the twisted homogeneous rack
    $\hC'_{\ell'}$ of type $(\Alt_{\{3, \dots, n\}}, t, \id)$
    is a subrack of $\hC_\ell$
    for $\ell'$ of cycle type $(1^{r_1 + 2}, 2^{r_2 - 1})$
    and leaving $1$ and $2$ fixed.
    Thus, the claim follows from \Cref{prop:id_12r}.
\end{proof}

\subsection{Proof of Main Result}

In this section, we use the results of
\Cref{sec:id} and \Cref{sec:iota} to prove \Cref{thm:main}.

\begin{proof}[Proof of \Cref{thm:main}]
    All cases not listed in
    \Cref{tab:tw_hom_a_n_type_d_solved} or
    \Cref{tab:tw_hom_a_n_type_d_unknown}
    are resolved
    by~\cite{andruskiewitsch2010twisted}*{Theorem~1.2}.
    The case listed in the first section
    of \Cref{tab:tw_hom_a_n_type_d_solved}
    with $\ell = e$ is resolved by \Cref{prop:id_1},
    the cases in this section
    with $\ell$ an involution
    are resolved by \Cref{prop:id_12r},
    and the last case listed in this section
    is resolved by \Cref{prop:id_124} and \Cref{prop:id_14}.
    The first two cases listed in the second section
    of \Cref{tab:tw_hom_a_n_type_d_solved}
    are resolved by \Cref{prop:iota_12},
    the case with $\ell (1\ 2)$ of type $(2^3)$
    is resolved by \Cref{prop:iota_222},
    and the last three cases listed in
    \Cref{tab:tw_hom_a_n_type_d_solved}
    are resolved by \Cref{prop:iota_12r}.
    The results which show each case to be of type D
    are also listed in \Cref{tab:tw_hom_a_n_type_d_solved}.
    
    This leaves only the racks listed in
    \Cref{tab:tw_hom_a_n_type_d_unknown}
    not known to be of type D.
    Therefore, any twisted homogeneous rack over 
    a simple alternating group that is not of type D
    is in the list given in the statement of \Cref{thm:main}.
\end{proof}

\section*{Acknowledgements}

The author would like to thank
Dr.~Julia Plavnik and Dr.~H\'ector Pe\~na Pollastri
for their mentorship,
and the MIT PRIMES research program
for making this work possible.
The author would also like to thank Dr.~Leandro Vendramin
for his valuable advice
on an earlier version of this manuscript,
and the Reviewer
for their thoughtful comments and suggestions.


\begin{bibdiv}
\begin{biblist}

\bib{ANDRUSKIEWITSCH201536}{article}{
    author={Andruskiewitsch, Nicol{\'a}s},
    author={Carnovale, Giovanna},
    author={{Ga}rc{\'\i}a, Gast{\'o}n~Andr{\'e}s},
    title={Finite-dimensional pointed {H}opf algebras over finite simple groups of {L}ie type I. {N}on-semisimple classes in $\operatorname{PSL}_n(q)$},
    journal={Journal of Algebra},
    volume={442},
    pages={36--65},
    year={2015},
    issn={0021-8693},
}

\bib{andruskiewitsch2010twisted}{article}{
      author={Andruskiewitsch, Nicol{\'a}s},
      author={Fantino, Fernando},
      author={{Ga}rc{\'\i}a, Gast{\'o}n~Andr{\'e}s},
      author={Vendramin, Leandro},
       title={On twisted homogeneous racks of type {D}},
        date={2010},
     journal={Revista de la Uni\'on Matem\'atica Argentina},
      volume={51},
      number={2},
       pages={1\ndash 16},
}

\bib{andruskiewitsch2011nichols}{article}{
      author={Andruskiewitsch, Nicol{\'a}s},
      author={Fantino, Fernando},
      author={{Ga}rc{\'\i}a, Gast{\'o}n~Andr{\'e}s},
      author={Vendramin, Leandro},
       title={On {N}ichols algebras associated to simple racks},
        date={2011},
     journal={Contemp. Math},
      volume={537},
       pages={31\ndash 56},
}

\bib{andruskiewitsch2011finite}{article}{
      author={Andruskiewitsch, Nicol{\'a}s},
      author={Fantino, Fernando},
      author={Gra{\~n}a, Mat{\'\i}as},
      author={Vendramin, Leandro},
       title={Finite-dimensional pointed {H}opf algebras with alternating
  groups are trivial},
        date={2011},
     journal={Annali di Matematica Pura ed Applicata},
      volume={190},
      number={2},
       pages={225\ndash 245},
}

\bib{ANDRUSKIEWITSCH2011305}{article}{
      author={Andruskiewitsch, Nicol{\'a}s},
      author={Fantino, Fernando},
      author={Gra{\~n}a, Mat{\'\i}as},
      author={Vendramin, Leandro},
       title={Pointed {H}opf algebras over the sporadic simple groups},
        date={2011},
        ISSN={0021-8693},
     journal={Journal of Algebra},
      volume={325},
      number={1},
       pages={305\ndash 320},
  url={https://www.sciencedirect.com/science/article/pii/S0021869310005326},
}

\bib{ANDRUSKIEWITSCH2003177}{article}{
      author={Andruskiewitsch, Nicol{\'a}s},
      author={Gra{\~n}a, Mat{\'\i}as},
       title={From racks to pointed {H}opf algebras},
        date={2003},
        ISSN={0001-8708},
     journal={Advances in Mathematics},
      volume={178},
      number={2},
       pages={177\ndash 243},
  url={https://www.sciencedirect.com/science/article/pii/S0001870802000713},
}

\bib{andruskiewitsch2002pointed}{inproceedings}{
      author={Andruskiewitsch, Nicol{\'a}s},
      author={Schneider, Hans-J{\"u}rgen},
       title={Pointed {H}opf algebras},
organization={Cambridge University Press},
        date={2002},
   booktitle={New directions in {H}opf algebras},
       pages={1\ndash 68},
}

\bib{CARNOVALE}{article}{
    author={Carnovale, Giovanna},
    author={Garc\'ia~Iglesias, Agust\'in},
    year = {2016},
    pages = {193-218},
    title = {$\theta$-semisimple classes of type D in $\operatorname{PSL}_n(q)$},
    volume = {26},
    journal = {Journal of Lie theory}
}

\bib{fantino2013twisted}{article}{
      author={Fantino, Fernando},
      author={Vendramin, Leandro},
       title={On twisted conjugacy classes of type {D} in sporadic simple
  groups},
        date={2013},
     journal={Contemp. Math},
      volume={585},
       pages={247\ndash 259},
}

\bib{milinski2000pointed}{article}{
      author={Milinski, Alexander},
      author={Schneider, Hans-J{\"u}rgen},
       title={Pointed indecomposable {H}opf algebras over {C}oxeter groups},
        date={2000},
     journal={Contemporary Mathematics},
      volume={267},
       pages={215\ndash 236},
}

\bib{suzuki1982group}{book}{
      author={Suzuki, Michio},
       title={Group theory},
   publisher={Springer-Verlag},
        date={1982},
      number={1},
        ISBN={9783540109150},
}

\end{biblist}
\end{bibdiv}
\end{document}